\documentclass[10pt]{article}
\input{src/project-preamble}
\usepackage[display-counters=section]{declare-style-notes}
\usepackage{declare-maths}
\usepackage{declare-text}
\usepackage{declare-tables-figures}
\usepackage[numbering-within=section]{declare-theorems-notes}
\usepackage{declare-bibliography}
\usepackage{needspace}

\hypersetup{
	pdfauthor = {Nikita Nikolaev},
	pdftitle = {Invariant Flags for Singularly Perturbed Logarithmic Systems of Rank Two},
	pdfsubject = {Singular and exact perturbations and differential equations in the complex domain},
	pdfkeywords = {singular perturbation theory, exact perturbation theory, logarithmic differential systems, flag reduction, Levelt filtration, factorial asymptotics, Borel resummation}
}
\newcommand{\MSCSubjectCode}[1]{\href{https://zbmath.org/classification/?q=cc\%3A#1}{#1}}

\title{\Large\bfseries Invariant Flags for Singularly Perturbed\\
Logarithmic Systems of Rank Two}
\author{Nikita Nikolaev}
\affil{School of Mathematics, University of Edinburgh, United Kingdom}
\date{30 October 2026}

\begin{document}
\maketitle

\begin{abstract}\noindent
We study rank-two differential systems $\hbar x \partial_x \psi + A (x, \hbar) \psi = 0$ near a logarithmic pole, with coefficients that are holomorphic in $x$ and admit uniform factorial asymptotic expansions in the complex parameter $\hbar$.
If the classical residue has distinct eigenvalues, then upon choosing an ordering and an admissible $\hbar$ direction, we transform this system to an upper-triangular form by a holomorphic gauge transformation with the same factorial asymptotic regularity on a sector in the $\hbar$-plane.
Its diagonal is fully determined by the $\hbar \to 0$ and $x \to 0$ limits.
This construction produces a canonical $(x,\hbar)$-dependent invariant flag.
This flag specialises to the usual Levelt filtration for nonzero values of $\hbar$ on a narrower sector, and its associated graded converges to the eigenspaces of the residue as $x \to 0$ and of the leading-order matrix as $\hbar \to 0$.
A Schrödinger example gives a uniform matching formula between the WKB data and the exact Frobenius exponent at a double pole.
\end{abstract}

{\footnotesize

\textbf{2020 MSC:}
\MSCSubjectCode{34M60} (primary),
\MSCSubjectCode{34E10},
\MSCSubjectCode{34E20},
\MSCSubjectCode{40G10}.
}
\tableofcontents

%===============================================================================
\setcounter{section}{-1}
\section{Introduction}
%===============================================================================

We revisit the classical problem in singular perturbation theory of studying linear systems of ordinary differential equations of the form
\begin{eqn}
	\nabla \psi \coleq \hbar x \del_x \psi + A (x, \hbar) \psi = 0
\fullstop
\end{eqn}
Here, $x$ is a complex independent variable, $\hbar$ is a complex parameter, $\psi = \psi (x, \hbar)$ is a vector function, and $A (x, \hbar)$ is a matrix whose entries are holomorphic near $x = 0$ and admit uniform factorial asymptotic expansions as $\hbar \to 0$ in a sector.
For every fixed nonzero $\hbar$, this is an ordinary logarithmic differential system at $x = 0$, but the parameter limit $\hbar \to 0$ turns this into a singular perturbation problem.
An important class of examples is obtained when $A$ is polynomial in $\hbar$.
We restrict our attention to systems of rank $2$.

For every fixed nonzero $\hbar$, standard theory \cite{MR0460820} tells us that the finite-dimensional vector space of solutions $\V_\hbar$ carries a natural filtration by growth rates as $x \to 0$ which is sometimes called a \dfn{Levelt filtration} \cite{MR0145108,zoladek2006monodromy,MR2785493}.
In rank two, when the two growth rates are distinct, this filtration has a distinguished one-dimensional step, generated by any nonzero \textit{subdominant} solution:
\begin{eqn}
	\V^\bullet_\hbar
		\coleq \big( 0 \subset \V^1_\hbar \subset \V_\hbar \big)
\fullstop
\end{eqn}
The problem considered here is to construct this filtration as a holomorphic family and to make sense of its limits as $x \to 0$ and as $\hbar \to 0$.
Note that the solutions themselves generally have singular behaviour at both $x = 0$ and $\hbar = 0$, so taking limits of $\V^\bullet_\hbar$ as $x \to 0$ and as $\hbar \to 0$ directly does not necessarily make sense.
However, the corresponding filtration on the underlying trivial vector bundle $\E = \underline{\CC}^2$ over the $(x,\hbar)$-space has a much more regular behaviour, and it is this filtration that makes the two limits accessible.

\enlargethispage{10pt}
%---------------------------------------
\paragraph*{Main results}
Write $A_0 (x)$ for the leading-order matrix as $\hbar \to 0$, and put $A_{00} = A_0 (0)$.
We assume that $A_{00}$ has distinct eigenvalues $m_1, m_2$.
Then, near $x = 0$, the eigenvalues $\eta_i (x)$ of $A_0 (x)$ and the eigenvalues $\rho_i (\hbar)$ of $A (0, \hbar)$ can be labelled by $\eta_i (x) \to m_i$ as $x \to 0$ and $\rho_i (\hbar) \to m_i$ as $\hbar \to 0$.
Let us combine these eigenvalues into the following functions:
\begin{eqn}
	\lambda_i (x, \hbar) \coleq \eta_i (x) + \rho_i (\hbar) - m_i
\fullstop
\end{eqn}
Fix a direction $\theta$ such that $\Re \big( e^{-i \theta} (m_2 - m_1) \big) > 0$.
Our main result (\autoref{1793460680957}) constructs a regular gauge transformation $G = G(x, \hbar)$ which puts the system $\nabla$ into the upper-triangular form
\begin{eqn}
	\nabla \xmapsto{~~G~~} \tilde{\nabla} \coleq \hbar x \del_x 
	+ \mtx{ \lambda_1 & x\hbar u \\ 0 & \lambda_2 }
\fullstop{,}
\end{eqn}
for some regular function $u = u (x,\hbar)$.
Here, \textit{regular} means holomorphic in $(x, \hbar)$ with uniform factorial asymptotics as $\hbar \to 0$ in a sector with opening angle $\pi$ bisected by $\theta$ (\autoref{1793460818045}).
The upshot is that the first column of $\tilde{\nabla}$ defines through the gauge transformation $G$ a regular line $\L \subset \E$ which is preserved by $\nabla$; i.e., a $\nabla$-invariant flag:
\begin{eqn}
	0 \subset \L \subset \E
\qqtext{with}
	\nabla (\L) \subset \L
\qtext{and}
	\nabla (\E / \L) \subset \E / \L
\fullstop
\end{eqn}
We prove in \autoref{1793461788536} that $\L$ and $\E / \L$ tend respectively to the first and second eigenspaces of the residue $A (0, \hbar)$ as $x \to 0$ and to the first and second eigenspaces of the leading-order part $A_0 (x)$ as $\hbar \to 0$.
Finally, for every fixed nonzero $\hbar$ satisfying $\Re \big(\rho_1 (\hbar) / \hbar\big) < \Re \big(\rho_2 (\hbar) / \hbar\big)$, we show in \autoref{1793461989510} that the solutions of $\nabla$ taking values in $\L$ determine precisely the Levelt filtration $\V^\bullet_\hbar$.
Furthermore, the solutions of $\nabla$ that determine the invariant line $\L$ are essentially the Borel resummation of their formal expansions in $\hbar$.
We explain this relation in \ref{1793463810732} and work it out for a Schrödinger equation in \autoref{1793463856919}.
The proof uses a logarithmic Riccati lemma (\autoref{1793462388439}) by passing to logarithmic coordinates and applying the half-strip existence result of \cite[Lemma 5.1]{MY2008.06492}.
The main point is to descend the resulting solution from the logarithmic cover and extend it holomorphically across the pole.

%---------------------------------------
\paragraph*{Context}
Triangular reduction near a regular singular point has a long history, including Sibuya's analytic reduction theorem \cite{MR0096016} as well as the work of Russell and Sibuya \cite{MR0214869,MR0245929}.
Factorial asymptotics and summability for singular perturbations have been developed in \cite{Sibuya2000,MR1914445,MR1919789,MR2292516}.
Balser and Kostov \cite{MR1914445} prove factorial estimates and a summability theorem for an inhomogeneous linear logarithmic equation with coefficients holomorphic at the parameter origin.
Marchetti and Conti \cite{MarchettiConti2018} treat the nonlinear regular-singular problem with analytic input, including summability under a Siegel-type condition.
For general background on summability, see e.g. \cite{Balser1994}.

In particular, \cite[Theorem 6]{MR0214869} establishes a block upper-triangular reduction on a full disc in $x$, with a gauge transformation admitting an ordinary asymptotic expansion in the perturbation parameter.
Here we construct the canonical rank-two invariant flag for sectorial input with uniform factorial asymptotics, preserving this regularity across the logarithmic pole.
We identify this flag with the Borel resummation of its formal counterpart and determine the residue and leading-order spectral limits of its line and quotient.
On the region where the residue exponents have the prescribed growth order, the flag induces precisely the Levelt filtration of the specialised system.
These statements give the geometric content of the reduction and specify how the two singular limits interact.
The scalar consequence in \autoref{1793463856919} connects these limits with the exact WKB constructions of \cite{MY210623112236}.

%---------------------------------------
\paragraph*{Outline}
We fix the asymptotic conventions and prepare the spectral data in \autoref{1793620912250}.
The formal normal form and its factorial estimates are proved in \autoref{1793613963090}.
The logarithmic Riccati lemma and flag reduction occupy \autoref{1793614586143}.
We construct the invariant line and identify the growth filtration in \autoref{1793615066656}, and conclude with examples in \autoref{1793616097232}.

%---------------------------------------
\paragraph*{AI Declaration}
OpenAI Codex (GPT-6.1, model identifier \texttt{gpt-6.1-sol}) was used during the preparation of this article.
Its use was limited to editorial feedback, suggestions for wording, bibliographic assistance, and proofreading.
The conceptual framework and mathematical proofs are almost entirely the author's own work, with the author retaining full accountability for the final content.

%---------------------------------------
\paragraph*{Acknowledgements}
The author thanks Marco Gualtieri, Kohei Iwaki, and Shinji Sasaki for helpful discussions during the development of this work.
The original work was supported by the NCCR SwissMAP of the SNSF.
Several aspects were initiated while the author was a PhD student at the University of Toronto, whose Department of Mathematics provided a productive and supportive environment.

%===============================================================================
%===============================================================================
\section{Singularly Perturbed Systems}
\label{1793620912250}
%===============================================================================
%===============================================================================

%===============================================================================
\subsection{Domains and asymptotic regularity}
%===============================================================================

For any $R, r > 0$ and any direction $\theta \in \RR$, introduce
\begin{eqntag}
\label{1793464941667}
	\DD_R \coleq \set{ x \in \CC ~:~ |x| < R }
\qqtext{and}
	\SS_{\theta, r} \coleq \set{ \hbar \in \CC ~\Big|~ \Re (e^{i\theta}/\hbar) > 1/r }
\fullstop
\end{eqntag}
Thus, $\SS_{\theta, r}$ is the open disc of diameter $r$ tangent to the origin $\hbar = 0$ in the direction $\theta$; it is sometimes called a \dfn{Borel disc}.
It is a sectorial neighbourhood, or simply \dfn{sector} from now on, of the origin with opening $\A_\theta \coleq (\theta - \pi/2, \theta + \pi/2)$.

%---------------------------------------
\begin{definition}
\label{1793460818045}
A holomorphic function $f$ on $\DD_R \times \SS_{\theta, r}$ is \dfn{regular} if it admits a uniform factorial asymptotic expansion
\begin{eqn}
	f (x, \hbar) \sim \hat{f} (x, \hbar)
	= \sum_{n =0}^\infty f_n (x) \hbar^n
\fullstop
\end{eqn}
This means that the coefficients are holomorphic on $\DD_R$ and that there are constants $C, M > 0$ such that, for every $N \geq 0$,
\begin{eqntag}
\label{1793465750223}
	\left|f (x, \hbar) - \sum_{n = 0}^{N - 1}f_n (x) \hbar^n \right|
	\le C M^N N!|\hbar|^N
\fullstop{,}
\end{eqntag}
for all $(x, \hbar) \in \DD_R \times \SS_{\theta,r}$.
We identify functions after restriction to smaller such domains and denote the resulting ring of germs by $\cal{R}_\theta$, or simply $\cal{R}$.
\end{definition}

\begin{remark}
Some authors call these \textit{Gevrey asymptotics} (of order one).
We use the term \textit{factorial asymptotics} to make the factorial growth in \eqref{1793465750223} explicit.
\end{remark}

The uniformity in this definition includes approach to the boundary directions within the Borel disc.
It is stronger than asking for estimates only on each proper subsector.
If the coefficients are given on a larger sector with locally uniform factorial asymptotics, the definition applies after restriction to a Borel disc whose closed opening lies inside that sector.

%---------------------------------------
\begin{definition}
\label{1793465612390}
A \dfn{regular formal series} is a series $\hat{f} = \sum f_n (x) \hbar^n$ whose coefficients are holomorphic on a common disc and satisfy $|f_n (x)|\le CM^n n!$ there.
Germs of such series form the ring $\hat{\cal{R}}$.
\end{definition}

The asymptotic expansion in $\hbar$ defines a map $\cal{R} \to \hat{\cal{R}}$ sending $f \mapsto \hat{f}$.
Both rings $\cal{R}$ and $\hat{\cal{R}}$ are closed under sums and products, differentiation in $x$ on smaller discs, inversion of an element with nowhere-zero leading term, and holomorphic composition near a fixed leading value.
These are the usual closure properties of factorial asymptotics; see \cite{Balser1994,MY211216122156}.
For example, if $f_0 = 0$, division by $\hbar$ preserves regularity by shifting the remainder estimate in \eqref{1793465750223}.
If also $f (0, \hbar) = 0$, division by $x$ preserves regularity on a smaller disc by Cauchy's formula.
Consequently,
\begin{eqntag}
\label{1793465774259}
	f_0 = 0 \qtext{and} f (0, \hbar) = 0
	\quad \implies \quad f \in x \hbar \cal{R}
\fullstop
\end{eqntag}
Similarly for formal series:
\begin{eqntag}
\label{1793465827461}
	f_0 = 0 \qtext{and} \hat{f} (0, \hbar) = 0
	\quad \implies \quad \hat{f} \in x \hbar \hat{\cal{R}}
\fullstop
\end{eqntag}

%===============================================================================
\subsection{Systems and gauge transformations}
%===============================================================================

%---------------------------------------
\begin{definition}
\label{1793620942067}
A \dfn{singularly perturbed logarithmic system of rank two} is an equation
\begin{eqntag}
\label{1793620937922}
	\nabla \psi = 0,
	\qquad \nabla = \hbar x \del_x + A (x, \hbar),
	\qquad A \in \op{Mat}_2 (\cal{R})
\fullstop
\end{eqntag}
It is \dfn{generic} if its classical residue $A_{00} = A_0 (0)$ has distinct eigenvalues.
\end{definition}

%---------------------------------------
\begin{definition}
\label{1793465991744}
A \dfn{regular gauge transformation} is a matrix $G \in \op{GL}_2 (\cal{R})$.
A \dfn{regular formal gauge transformation} is a matrix $\hat{G} \in \op{GL}_2 (\hat{\cal{R}})$.
\end{definition}

We use the gauge transformation convention
\begin{eqntag}
\label{1793466110763}
	\tilde{\nabla} = G \bullet \nabla = G \nabla G^{-1} = x \hbar \del_x + \tilde{A} (x, \hbar)
\qtext{where}
	\tilde{A} = G \bullet A \coleq G A G^{-1} - \hbar x (\del_x G)G^{-1}
\fullstop
\end{eqntag}
Thus,
\begin{eqntag}
\label{1793466364419}
	\nabla \psi = 0
\qtext{if and only if}
	\tilde{\nabla} \tilde{\psi} = 0
\qtext{whenever}
	\tilde{\psi} = G \psi
\fullstop
\end{eqntag}

%===============================================================================
\subsection{Spectral preparation}
%===============================================================================

Let $m_1 \neq m_2$ be the eigenvalues of $A_{00}$.
Then $A_0 (x)$ has distinct holomorphic eigenvalues $\eta_i (x)$ on a sufficiently small disc $\DD_R$ and $A (0, \hbar)$ has distinct regular eigenvalues $\rho_i (\hbar)$ on a sufficiently small sector $\SS_{\theta,r}$.
Moreover,
\begin{eqn}
	\eta_i (0) = m_i
\qqtext{and}
	\rho_i (\hbar) \to m_i
\quad
	\text{as $\hbar \to 0$ in $\SS_{\theta,r}$}
\fullstop
\end{eqn}
The corresponding projectors
\begin{eqntag}
\label{1793610114339}
	\Pi_i (\hbar) =
	\frac{A (0, \hbar) - \rho_j (\hbar) \Id}{\rho_i (\hbar) - \rho_j (\hbar)}
\fullstop{,}
\qquad
	j \neq i
\fullstop{,}
\end{eqntag}
give regular eigenframes by application to classical eigenvectors.
This is the simple-spectrum case of diagonalisation in families with factorial asymptotics \cite[Theorem 1.2]{MY211216122156}; only existence is needed here.

We call the matched pairs $(\eta_i, \rho_i)$ the \dfn{spectral data} and set
\begin{eqntag}
\label{1793610179470}
	\lambda_i (x, \hbar) \coleq \eta_i (x) + \rho_i (\hbar) - m_i
\qqtext{and}
	\Lambda \coleq \op{diag} (\lambda_1, \lambda_2)
\fullstop
\end{eqntag}
We also set
\begin{eqntag}
\label{1793610823323}
	\hat{\lambda}_i (x, \hbar) \coleq \eta_i (x) + \hat{\rho}_i (\hbar) - m_i
\qqtext{and}
	\hat{\Lambda} \coleq \op{diag} (\hat{\lambda}_1, \hat{\lambda}_2)
\fullstop
\end{eqntag}
The matching is fixed by the distinct classical eigenvalues $m_1, m_2$, so that the only ambiguity is one simultaneous permutation.
Both $\eta_i$ and $\rho_i$ are invariant under regular gauge transformations: in the restrictions $\hbar = 0$ and $x = 0$, the derivative term in \eqref{1793466110763} vanishes.
Under a formal gauge equivalence the corresponding invariants are $\eta_i$ and the formal expansions $\hat{\rho}_i$.

%---------------------------------------
\begin{lemma}[spectral preparation]
\label{1793610475323}
Any generic system is equivalent via regular gauge transformations to a system of the form
\begin{eqntag}
\label{1793610385890}
	\nabla = x\hbar \del_x + A
\qtext{where}
	A = \Lambda + \hbar x B
\qquad
	\text{for some $B \in \op{Mat}_2 (\cal{R})$}
\fullstop
\end{eqntag}
Similarly, any generic system is equivalent via regular formal gauge transformations to a system of the form
\begin{eqntag}
\label{1793610758701}
	\hat{\nabla} = x\hbar \del_x + \hat{A}
\qtext{where}
	\hat{A} = \hat{\Lambda} + \hbar x \hat{B}
\qquad
	\text{for some $\hat{B} \in \op{Mat}_2 (\hat{\cal{R}})$}
\fullstop
\end{eqntag}
\end{lemma}

\begin{proof}
Fix a constant matrix $P_{00}$ diagonalising $A_{00}$.
Choose a holomorphic matrix $P (x)$ diagonalising $A_0 (x)$ with $P (0) = P_{00}$, and a regular matrix $H (\hbar)$ diagonalising $A (0, \hbar)$ with $H_0 = P_{00}$.
Thus, $P A_0 P^{-1} = \diag (\eta_1, \eta_2)$ and $H A (0,\hbar) H^{-1} = \diag (\rho_1, \rho_2)$.
The latter exists by \cite[Theorem 1.2]{MY211216122156} using the projectors \eqref{1793610114339}; the former is constructed in the same way.
Put $G = P (x) P_{00}^{-1}H (\hbar)$.
Then $G_0 (x) = P (x)$ and $G (0, \hbar) = H (\hbar)$.
The transformed matrix $G \bullet A$ therefore has leading-order matrix $\op{diag} (\eta_1, \eta_2)$ and residue $\op{diag} (\rho_1, \rho_2)$.
Its difference from $\Lambda$ vanishes in both restrictions, so \eqref{1793465774259} proves the assertion.
Taking asymptotic expansions gives the formal statement.
\end{proof}

%===============================================================================
\clearpage
\section{The Formal Normal Form}
\label{1793613963090}
%===============================================================================
We first construct the formal normal form of a generic system.
Unlike the regular construction, this does not require a choice of direction in the $\hbar$-plane.
We begin with a formal Riccati equation and prove the factorial estimates needed for its solution to be regular.

%---------------------------------------
\begin{lemma}[formal logarithmic Riccati equation]
\label{1793614188268}
Let $\hat{a}, \hat{b}, \hat{c} \in \hat{\cal{R}}$ satisfy $a_0 = c_0 = 0$ and $b_0 (0) \neq 0$.
Then the equation
\begin{eqntag}
\label{1793614194918}
	x \hbar \del_x \hat{s}
		= \hat{a} + \hat{b} \hat{s} + \hat{c} \hat{s}^2
\fullstop
\end{eqntag}
has a unique regular formal solution $\hat{s}$ with $s_0 = 0$.
Its coefficients are holomorphic on a common disc and satisfy uniform factorial bounds on every smaller disc.
Consequently, $\hat{s} \in \hbar \hat{\cal{R}}$.
Moreover, if $\hat{a} (0, \hbar) = 0$, then $\hat{s} \in x \hbar \hat{\cal{R}}$.
\end{lemma}

\begin{proof}
Choose $R$ so small that $b_0$ is bounded away from zero on $\DD_R$ and there exist constants $A, B > 0$ such that
\begin{eqn}
	\sup_{\DD_R} \big( |a_n| + |b_n| + |c_n| \big)
		\le A B^n n!
\fullstop
\end{eqn}
Comparing coefficients of $\hbar^n$ uniquely determines $s_n$ for every $n \geq 1$:
\begin{eqntag}
\label{1793614178029}
	s_n = b_0^{-1}
		\left(
			x \del_x s_{n - 1} - a_n
			- \sum_{i = 1}^{n - 1} b_i s_{n - i}
			- \sum_{\substack{i + j + k = n \\ i, j, k \geq 1}} c_i s_j s_k
		\right)
\fullstop
\end{eqntag}
Thus every coefficient $s_n$ is holomorphic on the same disc $\DD_R$.
If each $a_n (0) = 0$, induction gives $s_n (0) = 0$.

To obtain factorial estimates on a common disc, put $d (x) = R - |x|$ and introduce the Nagumo norms \cite{CRSS2000}:
\begin{eqn}
	\|f\|_p \coleq \sup_{\DD_R} |f (x)| d (x)^p
\qqtext{for}
	p \geq 0
\fullstop
\end{eqn}
They satisfy
\begin{eqntag}
\label{1793614170470}
	\|fg\|_{p + q} \leq \|f\|_p \|g\|_q
\qqtext{and}
	\|\del_x f\|_{p + 1} \leq e (p + 1) \|f\|_p
\fullstop
\end{eqntag}
For $p > 0$, the second estimate follows from Cauchy's formula on a circle of radius $d (x) / (p + 1)$ and the inequality $(1 + 1/p)^p \leq e$.
For $p = 0$, take radii tending to $d (x)$.
Let $C_0 = \|b_0^{-1}\|_0$ and choose $\beta \geq \max (1, BR)$.
Then $\|a_n\|_n, \|b_n\|_n, \|c_n\|_n \leq A \beta^n n!$.
With $u_n = \|s_n\|_n / n!$ and $u_0 = 0$, equations \eqref{1793614178029} and \eqref{1793614170470} imply
\begin{eqn}
	u_n \leq C_0
		\left(
			e R u_{n - 1} + A \beta^n
			+ A \sum_{i = 1}^{n - 1} \beta^i u_{n - i}
			+ A \sum_{\substack{i + j + k = n \\ i, j, k \geq 1}} \beta^i u_j u_k
		\right)
\fullstop
\end{eqn}
Here we used the factorial inequalities $i! j! \leq (i + j)!$ and $i! j! k! \leq (i + j + k)!$.

Put $a (t) = A \beta t / (1 - \beta t)$ and let $v (t)$ be the formal series with zero constant coefficient determined by
\begin{eqn}
	v = C_0 \big( e R t v + a (t) (1 + v + v^2) \big)
\fullstop
\end{eqn}
Its coefficients are nonnegative and majorise the $u_n$.
By the Holomorphic Implicit Function Theorem at $(t, v) = (0, 0)$, the series $v$ converges.
Thus $u_n \leq C K^n$ for some $C, K > 0$, and on $\DD_r$ with $r < R$ we obtain
\begin{eqn}
	\sup_{\DD_r} |s_n|
		\leq C \left( \frac{K}{R - r} \right)^n n!
\fullstop
\end{eqn}
This is the required factorial estimate.
When all $a_n (0) = 0$, the solution coefficients vanish at zero, and division by $x$ on a smaller disc preserves this bound.
\end{proof}

%---------------------------------------
\begin{theorem}[formal normal form]
\label{1793614314431}
Any generic system is equivalent via regular formal gauge transformations to the diagonal system
\begin{eqn}
	\hat{\nabla} = x \hbar \del_x + \hat{\Lambda}
\fullstop{,}
\end{eqn}
where $\hat{\Lambda}$ is determined by \eqref{1793610823323}.
Two generic systems are formally equivalent if and only if their matched pairs $(\eta_i, \hat{\rho}_i)$ agree up to a simultaneous permutation.
\end{theorem}

\begin{proof}
By spectral preparation (\autoref{1793610475323}), we may assume that $\hat{A} = \hat{\Lambda} + x \hbar \hat{B}$.
Write $\hat{A} = (\hat{a}_{ij})$ and consider the two formal Riccati equations
\begin{eqntag}
\label{1793614752602}
\begin{aligned}
	x \hbar \del_x \hat{p}
		&= \hat{a}_{12} + (\hat{a}_{22} - \hat{a}_{11}) \hat{p} - \hat{a}_{21} \hat{p}^2
\fullstop{,}
\\
	x \hbar \del_x \hat{q}
		&= \hat{a}_{21} + (\hat{a}_{11} - \hat{a}_{22}) \hat{q} - \hat{a}_{12} \hat{q}^2
\fullstop
\end{aligned}
\end{eqntag}
The leading-order linear coefficients are $\eta_2 - \eta_1$ and $\eta_1 - \eta_2$, which are nonzero on a sufficiently small disc.
Moreover, $\hat{a}_{12}, \hat{a}_{21} \in x \hbar \hat{\cal{R}}$.
Thus \autoref{1793614188268} gives unique solutions $\hat{p}, \hat{q} \in x \hbar \hat{\cal{R}}$ with zero leading terms.
Put
\begin{eqn}
	\hat{F} = \mtx{ 1 & \hat{p} \\ \hat{q} & 1 }
\fullstop{,}
\qquad
	\hat{d}_1 = \hat{a}_{11} + \hat{a}_{21} \hat{p}
\fullstop{,}
\qquad
	\hat{d}_2 = \hat{a}_{22} + \hat{a}_{12} \hat{q}
\fullstop
\end{eqn}
The matrix $\hat{F}$ is invertible because its leading-order term is the identity.
Direct multiplication using \eqref{1793614752602} gives
\begin{eqn}
	\hat{F} \hat{A} - x \hbar \del_x \hat{F}
		= \diag (\hat{d}_1, \hat{d}_2) \hat{F}
\fullstop
\end{eqn}
Consequently, the gauge transformation $\hat{F}$ puts the system into diagonal form:
\begin{eqn}
	\hat{F} \bullet \hat{\nabla}
		= x \hbar \del_x + \diag (\hat{d}_1, \hat{d}_2)
\fullstop
\end{eqn}
It remains to bring the diagonal entries to the prescribed normal form.
Spectral preparation gives $\hat{d}_i - \hat{\lambda}_i \in x \hbar \hat{\cal{R}}$, so we may define a diagonal regular formal gauge transformation $\hat{K} = \diag (\hat{k}_1, \hat{k}_2)$ by
\begin{eqntag}
\label{1793614743610}
	\hat{k}_i = \exp \left(
		\int_0^x
			\frac{\hat{d}_i (t, \hbar) - \hat{\lambda}_i (t, \hbar)}{t \hbar}
		\, dt
	\right)
\fullstop
\end{eqntag}
Division by $x \hbar$, coefficientwise integration, and exponentiation preserve regular formal series on smaller discs.
For exponentiation, separate the holomorphic constant coefficient; the remaining series has positive $\hbar$-order, and the product factorial inequalities give a geometric majorant.
Thus each $\hat{k}_i$ is a unit with a nowhere-zero leading term.
Since $x \hbar \del_x \hat{k}_i = (\hat{d}_i - \hat{\lambda}_i) \hat{k}_i$, the gauge transformation $\hat{K}$ puts the system into the required form $x \hbar \del_x + \hat{\Lambda}$.
Composing these gauge transformations with the spectral preparation proves existence.

Conversely, a regular formal gauge transformation conjugates the leading-order matrices at $\hbar = 0$ and the formal residue matrices at $x = 0$.
It therefore preserves their eigenvalues, and the distinct classical eigenvalues $m_i$ fix the matching.
Hence the matched pairs $(\eta_i, \hat{\rho}_i)$ are necessary invariants.
They are also sufficient: systems with the same matched pairs reduce to the same diagonal system by the construction above.
\end{proof}

Thus the formal normal form is determined by the matched pairs $(\eta_i, \hat{\rho}_i)$.
For regular gauge equivalence, the corresponding spectral data involve the functions $\rho_i$ themselves.

%===============================================================================
\section{Flag Reduction}
\label{1793614586143}
%===============================================================================

We now move from formal to regular gauge transformations.
Our aim is to put the system into an upper-triangular form whose diagonal entries are the spectral data $\lambda_i$.
The main step is to solve a logarithmic Riccati equation with uniform factorial asymptotics on a full disc in the $x$-plane.

%===============================================================================
\subsection{A logarithmic Riccati lemma}
%===============================================================================

The following lemma solves the Riccati equation that arises after spectral preparation, with the solution extending holomorphically across the logarithmic pole.

%---------------------------------------
\begin{lemma}[logarithmic Riccati equation]
\label{1793462388439}
Let $a, b, c$ be regular on $\DD_R \times \SS_{\theta, r}$.
Assume
\begin{eqntag}
\label{1793614711206}
	a_0 = c_0 = 0
\fullstop{,}
\qquad
	a (0, \hbar) = c (0, \hbar) = 0
\fullstop{,}
\qquad
	\Re \big( e^{-i\theta} b_0 (0) \big) < 0
\fullstop
\end{eqntag}
Then, on $\DD_{R_0} \times \SS_{\theta, r_0}$ for sufficiently small $R_0, r_0 > 0$, the equation
\begin{eqntag}
\label{1793614719583}
	x \hbar \del_x s = a + b s + c s^2
\fullstop
\end{eqntag}
has a unique solution $s \in \cal{R}$ with $s_0 = 0$.
It satisfies $s (0, \hbar) = 0$ and $s \in x \hbar \cal{R}$.
\end{lemma}

\begin{proof}
We pass to logarithmic coordinates and apply \cite[Lemma 5.1]{MY2008.06492}.
We then show that the resulting solution is single-valued and extends across the pole.
That lemma gives a solution with uniform factorial asymptotics and zero leading term for an equation of the form
\begin{eqntag}
\label{1793614695949}
	\epsilon \del_z S
		= S + \epsilon \big( A_0 + A_1 S + A_2 S^2 \big)
\fullstop
\end{eqntag}
on a right half-strip and a smaller Borel disc, when the $A_j$ have uniform factorial bounds there.
The solution may be restricted to a smaller half-strip.

%------
\caps{Step 1: Logarithmic coordinates.}
Choose the disc sufficiently small that $b_0$ is nowhere zero, put $b_{00} = b_0 (0)$, and introduce the coordinate
\begin{eqntag}
\label{1793620976897}
	y = x \exp \left( \int_0^x \frac{b_0 (t) - b_{00}}{b_{00} t} \, dt \right)
\fullstop
\end{eqntag}
Then $y' (0) = 1$ and $b_0 (x) \, dx / x = b_{00} \, dy / y$.
Let $x = \chi (y)$ be its inverse on a sufficiently small disc $\DD_{R_1}$, with image compactly contained in $\DD_R$.
Set
\begin{eqn}
	\beta = e^{-i\theta} b_{00}
\fullstop{,}
\qquad
	\epsilon = e^{-i\theta} \hbar
\fullstop{,}
\qquad
	z = \beta \log y
\fullstop
\end{eqn}
Here $\Re (\beta) < 0$, and the deck period of the logarithmic coordinate is $\tau = 2\pi i \beta$.
Equation \eqref{1793614719583} becomes \eqref{1793614695949}, with
\begin{eqntag}
\label{1793620991595}
	A_0 = \frac{a}{\epsilon b_0}
\fullstop{,}
\qquad
	A_1 = \frac{b - b_0}{\epsilon b_0}
\fullstop{,}
\qquad
	A_2 = \frac{c}{\epsilon b_0}
\fullstop{,}
\end{eqntag}
evaluated at $(x, \hbar) = \big( \chi (e^{z/\beta}), e^{i\theta} \epsilon \big)$.
These coefficients have uniform factorial bounds wherever $|e^{z/\beta}| < R_1$; division by $\epsilon$ is justified by the vanishing of their numerators' leading terms.

%------
\caps{Step 2: One half-strip and its translates.}
Choose widths $w > w_0 > \tfrac12 |\Im (\tau)|$ and let
\begin{eqn}
	\H_w = \set{ z ~\Big|~ \op{dist} (z - T, \RR_{\geq 0}) < w }
\fullstop{,}
\end{eqn}
where $T$ is a sufficiently large positive real number.
If $z = T + t + v$ with $t \geq 0$ and $|v| < w$, then
\begin{eqn}
	\Re (z / \beta)
		\le T \Re (1 / \beta) + w / |\beta|
\fullstop
\end{eqn}
Since $\Re (1 / \beta) < 0$, we can ensure $|e^{z/\beta}| < R_1$ throughout $\H_w$.
The half-strip lemma, after translation by $T$, therefore gives a solution $S$ on $\H_{w_0} \times \SS_{0, r_0}$ with a uniform factorial expansion whose leading term is zero.

The coefficients of \eqref{1793614695949} and its selected formal solution are $\tau$-periodic wherever defined on the logarithmic cover: this is immediate for the coefficients and follows by formal recursion for the solution.
On $\H_{w_0} + k\tau$, define $S_k (z, \epsilon) = S (z - k\tau, \epsilon)$.
These translates solve the same equation and have the same formal expansion, with the same factorial bounds on the same sector.
On any intersection, their difference has zero factorial asymptotic expansion in $\epsilon$.
Pointwise uniqueness for uniform factorial asymptotics on a Borel disc (the Nevanlinna uniqueness used in \cite[proof of Lemma 5.1]{MY2008.06492}) gives equality there.
Thus they glue to a $\tau$-periodic holomorphic solution on $\W = \bigcup_{k \in \ZZ} (\H_{w_0} + k\tau)$, with the same uniform estimates.
The constants are uniform because all translates come from the same solution on $\H_{w_0}$.

%------
\caps{Step 3: Coverage of a punctured disc.}
Write $\beta = u + iv$, where $u < 0$.
The real linear function
\begin{eqn}
	J (z) = \Re (z) + \frac{v}{u} \Im (z)
\end{eqn}
is invariant under translation by $\tau$, and for $z = \beta \log y$ it satisfies
\begin{eqn}
	J (z) = \frac{|\beta|^2}{u} \log |y|
\fullstop
\end{eqn}
For every lift $z$, choose $k \in \ZZ$ with $|\Im (z - k\tau)| \le \tfrac12 |\Im (\tau)| < w_0$.
As $|y| \to 0$, $J (z) \to +\infty$, independently of the angular lift.
The real part of this chosen representative therefore exceeds $T$ for sufficiently small $|y|$.
It follows that every lift of a full punctured disc $0 < |y| < R_2$ lies in $\W$.
Periodicity descends the solution to that punctured disc, with estimates uniform in $y$ and $\epsilon$.

%------
\caps{Step 4: Extension and vanishing.}
The order-zero bound makes the solution bounded uniformly near $y = 0$, so it extends jointly holomorphically across the pole.
This follows, for example, from the Laurent expansion in $y$, whose negative coefficients vanish by boundedness, locally uniformly in the parameter.
The formal coefficients extend holomorphically across $x = 0$ by the recursion in \autoref{1793614188268}.
Every factorial remainder therefore extends with the same bound.
Returning to $x$ and restricting to a smaller disc gives the required regular solution.
Its zero leading term also gives $s = O (\hbar)$ uniformly.
Evaluating \eqref{1793614719583} at $x = 0$ yields $0 = b (0, \hbar) s (0, \hbar)$ by \eqref{1793614711206}.
After decreasing $r_0$, $b (0, \hbar)$ is nonzero, so $s (0, \hbar) = 0$.
Equation \eqref{1793465774259} now proves $s \in x \hbar \cal{R}$.

Finally, any two regular solutions with $s_0 = 0$ have the same formal expansion by \autoref{1793614188268}.
Their difference has zero uniform factorial asymptotic expansion on a common smaller Borel disc and vanishes by the same pointwise uniqueness theorem.
\end{proof}

%===============================================================================
\subsection{The triangular form}
%===============================================================================

Next, we apply the logarithmic Riccati lemma to eliminate the lower-left entry of the prepared system.
A further diagonal gauge transformation then puts the diagonal entries into the prescribed form $\lambda_1, \lambda_2$.
The resulting upper-triangular system is adapted to an invariant flag.

%---------------------------------------
\begin{theorem}[flag reduction]
\label{1793460680957}
Let $\nabla = x \hbar \del_x + A$ be a generic system on $\DD_R \times \SS_{\theta, r}$, with the classical eigenvalues labelled so that
\begin{eqntag}
\label{1793615611303}
	\Re \big (e^{- i\theta} (m_1 - m_2) \big) < 0
\fullstop
\end{eqntag}
Then there are sufficiently small $R_0, r_0 > 0$ and a regular gauge transformation $G$ on $\DD_{R_0} \times \SS_{\theta, r_0}$ which puts $\nabla$ into the upper-triangular form
\begin{eqntag}
\label{1793621006703}
	\tilde{\nabla} = G \bullet \nabla = x \hbar \del_x + T
\qtext{where}
	T = \mtx{ \lambda_1 & x \hbar u \\ 0 & \lambda_2 }
\fullstop{,}
\end{eqntag}
for some regular function $u$, where the $\lambda_i$ are the functions defined in \eqref{1793610179470}.
The prescribed opening angle $\pi$ is retained.
\end{theorem}

\begin{proof}
By spectral preparation (\autoref{1793610475323}), we may assume that $A = \Lambda + x \hbar B$.
Write $A = (a_{ij})$ and consider the lower-unipotent gauge transformation
\begin{eqn}
	G_1 = \mtx{ 1 & 0\\s & 1 }
\fullstop
\end{eqn}
The lower-left entry of the transformed matrix vanishes if and only if
\begin{eqntag}
\label{1793621021721}
	x \hbar \del_x s
	= a_{21} + (a_{11} - a_{22})s - a_{12} s^2
\fullstop
\end{eqntag}
The entries $a_{21}$ and $a_{12}$ belong to $x \hbar \cal{R}$, and the leading-order linear coefficient is $\eta_1 (x) - \eta_2 (x)$.
Its value at zero satisfies \eqref{1793615611303}.
Moreover, $a_{21} (0, \hbar) = a_{12} (0, \hbar) = 0$.
Thus \autoref{1793462388439} gives a regular solution $s \in x \hbar \cal{R}$ on a smaller domain with the same opening.
Direct multiplication gives
\begin{eqn}
	G_1 \bullet \nabla = x \hbar \del_x + \Lambda + V
\qtext{where}
	V = \mtx{ v_{11} & a_{12} \\ 0 & v_{22} }
\fullstop{,}
\end{eqn}
where
\begin{eqntag}
\label{1793621029619}
	v_{11} = a_{11} - \lambda_1 - a_{12} s
\fullstop{,}
\qquad
	v_{22} = a_{22} - \lambda_2 + a_{12} s
\fullstop
\end{eqntag}
The preparation and the vanishing of $s$ imply $v_{ii} \in x \hbar \cal{R}$.
Define
\begin{eqntag}
\label{1793616119128}
	g_i (x, \hbar) = \exp \left (
	\int_0^x \frac{v_{ii} (t, \hbar)}{t \hbar}\, dt \right)
\fullstop
\end{eqntag}
These functions are regular units.
The diagonal gauge transformation $G_2 = \diag (g_1, g_2)$ satisfies $x \hbar \del_x g_i = v_{ii} g_i$.
It therefore puts the system into the form \eqref{1793621006703}, with
\begin{eqn}
	u = \frac{a_{12} g_1 g_2^{-1}}{x \hbar} \in \cal{R}
\fullstop
\end{eqn}
Composing $G_2 G_1$ with the preparation gauge proves the theorem.
\end{proof}

We call this construction a \dfn{flag reduction}, since the transformed system is adapted to an invariant flag.
The gauge transformation depends on the choice of eigenframes and their normalisations, but the invariant line will be independent of these choices.

%===============================================================================
\section{The Invariant Line and the Growth Filtration}
\label{1793615066656}
%===============================================================================

We now describe the geometric meaning of the flag reduction.
It determines an invariant line whose limits as $x \to 0$ and as $\hbar \to 0$ recover the corresponding eigenspaces of the residue and the leading-order matrix.
We show that this line is canonical by identifying its solutions with the subdominant step of the Levelt filtration, and explain its relation to Borel resummation.

Let $\E = \underline{\CC}^2$ be the underlying trivial vector bundle over $\DD_{R_0} \times \SS_{\theta, r_0}$, equipped with the differential operator $\nabla = x \hbar \del_x + A$.
A $\nabla$-invariant line is a holomorphic line subbundle $\L \subset \E$ such that $\nabla$ preserves its local sections.
Equivalently, for any local generator $e_1$ of $\L$, there is a holomorphic function $\alpha_1$ such that $\nabla e_1 = \alpha_1 e_1$.
Thus $\nabla$ restricts to a rank-one differential operator on $\L$ and induces a rank-one differential operator $\bar{\nabla}$ on the quotient line bundle $\Q \coleq \E / \L$.
The latter is defined by $\bar{\nabla} \bar{v} = \overline{\nabla v}$, independently of the choice of lift $v$ of a local section $\bar{v}$ of $\Q$.
Both operators satisfy the same Leibniz rule as $\nabla$.

We call the line $\L$ \dfn{regular} if it is generated by the first vector of a frame of $\E$ whose matrix belongs to $\op{GL}_2 (\cal{R})$.
The quotient then inherits a regular frame from the second vector.

The upper-triangular form constructed above gives precisely such an invariant line: in the transformed frame, it is the line generated by the first basis vector.
The following theorem describes this line and its quotient in the original frame, together with their limits as $x \to 0$ and as $\hbar \to 0$.

%---------------------------------------
\begin{theorem}
\label{1793461788536}
Under the assumptions of \autoref{1793460680957}, there is a regular invariant line $\L \subset \E$ and a regular quotient $\Q = \E/\L$.
They admit generators $e_1$ and $\bar{e}_2$ satisfying
\begin{eqntag}
\label{1793615286448}
	\nabla e_1 = \lambda_1 e_1
\qqtext{and}
	\bar{\nabla} \bar{e}_2 = \lambda_2 \bar{e}_2
\fullstop
\end{eqntag}
The limits as $x \to 0$ and as $\hbar \to 0$ are described as follows:
\begin{enumerate}
\item $\L (0, \hbar)$ is the $\rho_1 (\hbar)$-eigenspace of $A (0, \hbar)$.
The quotient $\Q (0, \hbar)$ is canonically isomorphic to its $\rho_2 (\hbar)$-eigenspace.
\item The leading-order line $\L_0 (x)$ is the $\eta_1 (x)$-eigenspace of $A_0 (x)$.
The leading-order quotient $\Q_0 (x)$ is canonically isomorphic to its $\eta_2 (x)$-eigenspace.
\end{enumerate}
The limits are uniform on smaller discs and the parameter domains specified by regularity.
At $(x, \hbar) = (0, 0)$, they recover the two eigenspaces of the classical residue $A_{00}$.
\end{theorem}

\begin{proof}
Let $G$ be the gauge of \autoref{1793460680957} and let $\varepsilon_1, \varepsilon_2$ be the standard basis of $\CC^2$.
Set $e_i = G^{-1} \varepsilon_i$ and let $\L$ be generated by $e_1$.
Since $G$ and its inverse are regular, the $e_i$ are a regular frame, with a nondegenerate leading-order frame.
By the gauge transformation formula \eqref{1793466110763},
\begin{eqn}
	\nabla e_1 = \lambda_1 e_1
\fullstop{,}
\qquad
	\nabla e_2 = \lambda_2 e_2 + x \hbar u e_1
\fullstop
\end{eqn}
This proves invariance and, after taking the quotient, the asserted equation for $\bar{e}_2$.

At $x = 0$, both the derivative term in the gauge transformation formula and the off-diagonal entry $x \hbar u$ vanish.
Thus $G (0, \hbar)$ conjugates the residue to $\diag (\rho_1, \rho_2)$.
Taking leading-order terms as $\hbar \to 0$ similarly gives $G_0 A_0 G_0^{-1} = \diag (\eta_1, \eta_2)$, since the off-diagonal entry has zero leading term.
The derivative term tends to zero by Cauchy's estimates on a smaller disc.
These identities give the asserted limits of $\L$.
Since the eigenvalues are distinct, the quotient projection restricts to an isomorphism from the second eigenspace to the quotient by the first eigenspace.
This identifies the limiting quotient fibres canonically with the second eigenspaces.
Uniformity follows from the regularity of the frame and its inverse; uniformity of the $x$-limit follows from Cauchy's estimates on smaller discs and the uniform boundedness on the parameter domain.
\end{proof}

%===============================================================================
\subsection{Local solution spaces}
%===============================================================================

Fix a simply connected punctured sector $\U \subset \DD_{R_0}^*$ and a branch of $\log x$ on it.
For a fixed nonzero $\hbar$, let $\V_\hbar$ be the two-dimensional complex vector space of solutions on $\U$.
As $\hbar$ varies, these spaces form a holomorphic rank-two bundle: evaluation at a fixed $x_*\in \U$ identifies each fibre with $\CC^2$.
The solutions taking values in $\L$ form a one-dimensional subspace $\V^1_\hbar \subset \V_\hbar$, giving a holomorphic family of flags
\begin{eqn}
	\V^\bullet_\hbar \coleq \big( 0 \subset \V^1_\hbar \subset \V_\hbar \big)
\fullstop
\end{eqn}
The flag is intrinsic; only its identification by evaluation depends on the basepoint $x_*$.
Monodromy preserves this flag, so it may equivalently be regarded as a filtration of the local system of solutions.

For fixed $\hbar$, put $\alpha_i (\hbar) = \Re (\rho_i (\hbar)/\hbar)$.
When $\alpha_1 < \alpha_2$, the subdominant step is the space of solutions satisfying
\begin{eqntag}
\label{1793615379926}
	|x|^{\alpha_1 + \delta}\|\psi (x)\|\longrightarrow 0
	\quad \text{for every } \delta > 0
\fullstop{,}
\end{eqntag}
as $x \to 0$ in closed subsectors of $\U$.
We use the term \dfn{Levelt filtration} for this filtration by growth rates; see \cite{MR0145108,MR0460820} for the regular-singular theory.

%---------------------------------------
\begin{theorem}[growth filtration]
\label{1793461989510}
Let $\L$ be the invariant line constructed in \autoref{1793461788536}, defined on $\DD_{R_0} \times \SS_{\theta, r_0}$.
For every $\hbar \in \SS_{\theta, r_0}$ satisfying
\begin{eqntag}
\label{1793615393525}
	\Re \big ((\rho_2 (\hbar) - \rho_1 (\hbar))/\hbar \big) > 0
\fullstop{,}
\end{eqntag}
the line $\V^1_\hbar \subset \V_\hbar$ is exactly the subdominant step \eqref{1793615379926}.
The invariant regular line $\L$ is uniquely determined by this property on a nonempty open parameter region and by holomorphic continuation to the connected parameter domain.
\end{theorem}

\begin{proof}
For this fixed $\hbar$, the scalar equations $x \hbar \del_x f_i + \lambda_i f_i = 0$ have nonzero solutions
\begin{eqntag}
\label{1793615545797}
	f_i = x^{ - \rho_i (\hbar)/\hbar}k_i (x, \hbar)
\fullstop{,}
\qquad
	k_i = \exp \left ( - \frac{1}{\hbar}
	\int_0^x \frac{\eta_i (t) - m_i}{t}\, dt \right)
\fullstop
\end{eqntag}
For fixed $\hbar$, each $k_i$ is holomorphic and nonvanishing at zero.
The solution $f_1 e_1$ spans $\V^1_\hbar$ and satisfies \eqref{1793615379926}.

For any solution $\psi$ outside this line, the second component of $G \psi$ is a nonzero constant multiple of $f_2$.
Since $G$ is bounded near $x = 0$ for fixed $\hbar$, the norm of $\psi$ is bounded below by a positive constant times $|f_2|$.
Choosing $0 < \delta < \alpha_2 - \alpha_1$ excludes \eqref{1793615379926}.
Thus $\V^1_\hbar$ is precisely the subdominant step of the Levelt filtration.

Condition \eqref{1793615611303} implies \eqref{1793615393525} for sufficiently small parameters along $\theta$, and on an open neighbourhood of that ray.
Thus the growth condition characterises the solution line on a nonempty open parameter set, and evaluation characterises $\L$ there.
Any two holomorphic line subbundles with this property agree on that open set and hence everywhere on the connected domain by the identity theorem for their maps to $\PP^1$.
\end{proof}

Thus, the growth condition makes the line $\L$ and the flag canonical.
A normalised generator is additional data, and no complementary invariant line is required.

%---------------------------------------
\begin{corollary}
\label{1793615627285}
Let $\J$ be a compact angular interval such that $\Re (e^{- i\varphi} (m_2 - m_1)) > 0$ for all $\varphi \in \J$.
On the part of the parameter domain with $\arg \hbar \in \J$, condition \eqref{1793615393525} holds for sufficiently small $|\hbar|$, uniformly in the direction.
\end{corollary}

\begin{proof}
Compactness gives $c > 0$ with $\Re (e^{- i\varphi} (m_2 - m_1)) \geq c$ on $\J$.
Regularity gives $\rho_2 - \rho_1 = (m_2 - m_1) + O (\hbar)$ uniformly.
The real part in \eqref{1793615393525} is therefore bounded below by $c/|\hbar| - C$ for a fixed constant $C$.
\end{proof}

%===============================================================================
\subsection{Continuation over an interval of directions}
%===============================================================================

The invariant line $\L$ is compatible with variation of the admissible direction $\theta$.

%---------------------------------------
\begin{corollary}[continuation in the parameter direction]
Let $\J = [\theta_-, \theta_+]$ be a compact angular interval such that
\begin{eqn}
	\Re \big( e^{- i\varphi} (m_2 - m_1) \big) > 0
\qqtext{for all}
	\varphi \in \J
\fullstop
\end{eqn}
Suppose that $A$ is holomorphic on a common disc times a sectorial domain and is regular on $\DD_R \times \SS_{\varphi, r_\varphi}$ for every $\varphi \in \J$.
After restricting the domains, the invariant lines constructed in these directions agree on overlaps and define a canonical invariant line $\L$ on a sectorial neighbourhood with opening
\begin{eqn}
	(\theta_- - \pi/2, \theta_+ + \pi/2)
\fullstop
\end{eqn}
The quotients $\Q$ and the spectral limits of $\L$ and $\Q$ agree on overlaps as well.
\end{corollary}

\begin{proof}
Choose finitely many directions in $\J$, including its endpoints, so that neighbouring Borel discs overlap in directions belonging to $\J$.
Apply the flag reduction theorem in each direction and decrease the spatial and parameter radii to common values.
On a nonempty open part of each neighbouring overlap, both invariant lines give the subdominant step of the Levelt filtration.
They therefore agree throughout the overlap by holomorphic continuation.
The lines consequently glue on the union of the chosen Borel discs, whose opening is the stated interval.
The same argument identifies the resulting line with the construction in any other direction in $\J$, after restriction.
The quotients and spectral limits glue with the lines, and uniform factorial estimates hold on each of the finitely many constituent domains.
\end{proof}

%===============================================================================
\subsection{Relation to Borel resummation}
\label{1793463810732}
%===============================================================================

The uniform factorial estimates give the invariant line a direct interpretation in terms of Borel resummation.

%---------------------------------------
\begin{proposition}[Borel resummation]
In the prepared frame of \autoref{1793610475323}, let $\hat{s}$ be the unique formal solution with zero leading term of the Riccati equation \eqref{1793621021721}.
Its Borel resummation $s$ in the direction $\theta$ generates the invariant line $\L$ through the column vector $(1, - s)^\top$.
This description is compatible with regular gauge transformations.

Moreover, fix a branch of $\log x$ on $\U$ and let $e_1$ be the generator of $\L$ constructed in \autoref{1793461788536}.
Set
\begin{eqntag}
\label{1793621050407}
	\Phi_1 (x) = m_1 \log x + \int_0^x \frac{\eta_1 (t) - m_1}{t}\, dt
\fullstop{,}
\qquad
	a_1 (x, \hbar) = x^{ - (\rho_1 (\hbar) - m_1)/\hbar} e_1 (x, \hbar)
\fullstop
\end{eqntag}
Then
\begin{eqn}
	\psi_1 = e^{ - \Phi_1/\hbar} a_1
\fullstop
\end{eqn}
is an exact solution spanning $\V^1_\hbar$.
On every compact subset of $\U$, the amplitude $a_1$ is the Borel resummation of its formal expansion in the direction $\theta$.
\end{proposition}

\begin{proof}
By Nevanlinna's theorem, a regular function is the Borel-Laplace resummation of its formal expansion, after restriction to a smaller Borel disc.
The estimates are uniform in $x$ on smaller spatial discs.
This is the resummation and uniqueness theorem used in \cite[Lemma 5.1 and its proof]{MY2008.06492}; see also \cite{Balser1994}.

The regular solution $s$ of \eqref{1793621021721} has formal expansion $\hat{s}$ by uniqueness of the formal solution with zero leading term.
It is therefore the Borel resummation of $\hat{s}$.
The vector $(1, - s)^\top$ generates the invariant line in the prepared frame.
Compatibility with regular gauge transformations follows because Borel resummation commutes with sums, products and inverses.

The expression for $\psi_1$ follows from \eqref{1793615545797}.
On every compact subset of $\U$, the quotient $(\rho_1 - m_1)/\hbar$ is regular and the chosen logarithm is bounded.
Thus $a_1$ is regular there, and another application of Nevanlinna's theorem proves the assertion about its Borel resummation.
\end{proof}

The amplitude estimates need not be uniform as $x \to 0$, whereas the invariant line $\L$ extends regularly across the pole.
On the region \eqref{1793615393525}, its solutions are exactly the subdominant solutions by \autoref{1793461989510}.

%===============================================================================
\section{Examples}
\label{1793616097232}
%===============================================================================

In this final section, we illustrate the preceding results with a few examples.
The first three clarify the relation between the invariant line, growth ordering, and the choice of a solution basis.
We then apply the construction to a Schrödinger equation with a double pole, obtaining a solution that combines the exact Frobenius exponent with uniform WKB asymptotics.

%---------------------------------------
\begin{example}[asymptotic opening and growth ordering]
\label{1793616106278}
Take $A = \diag (0, 1)$.
The two solutions are $\varepsilon_1$ and $x^{ - 1/\hbar} \varepsilon_2$.
The invariant line associated with the order $0, 1$ is the constant line $\CC \varepsilon_1$.
For parameters with $\Re (1/\hbar) > 0$, it is the subdominant line as $x \to 0$ along a fixed ray.
For parameters with $\Re (1/\hbar) < 0$, the second solution decays and has the slower growth.
At equality the two real growth rates coincide.
Thus the invariant line can continue holomorphically across a change in the ordering of the growth rates.
For example, gluing directions in $[ - \pi/4, \pi/4]$ produces an asymptotic opening $( - 3 \pi/4, 3 \pi/4)$, on which both signs occur.
\end{example}

%---------------------------------------
\begin{example}[mixed terms and diagonal normalisation]
\label{1793616110995}
For $m_1 \neq m_2$, let
\begin{eqn}
	A = \diag (m_1 + x \hbar, m_2)
\fullstop
\end{eqn}
Its spectral normal form is $\diag (m_1, m_2)$.
The diagonal gauge $G = \diag (e^x, 1)$ gives this form exactly by \eqref{1793466110763}.
This illustrates the diagonal normalisation \eqref{1793616119128}: divide the mixed term by $x \hbar$ and integrate.
\end{example}

%---------------------------------------
\begin{example}[a logarithmic term in a complementary solution]
\label{1793616125300}
Consider
\begin{eqn}
	A = \mtx{ 0 & x^n \hbar \\ 0 & 1 }, \qquad n \geq 1
\fullstop
\end{eqn}
The invariant line $\L$ and quotient $\Q$ are the constant coordinate line and its quotient, with no singularity at $x = 0$ or in their leading-order limits.
If we write a second solution as $x^{ - 1/\hbar} (\varepsilon_2 + c \varepsilon_1)$, then we get
\begin{eqn}
	x \hbar \del_x c = c - x^n \hbar
\fullstop
\end{eqn}
For $n \hbar \neq 1$, one particular solution is 
\begin{eqn}
	c = \frac{x^n \hbar}{1 - n \hbar}
\fullstop
\end{eqn}
At $\hbar = 1/n$, a particular solution is $c = - x^n \log x$.
The logarithm appears in the complementary solution, while the invariant line and quotient remain regular.
Thus the construction in \autoref{1793461788536} does not require a regular choice of a full solution basis.
\end{example}

%---------------------------------------
\begin{example}[a Schrödinger equation at a double pole]
\label{1793463856919}
Let $Q$ be holomorphic with $Q (0) \neq 0$ and consider
\begin{eqntag}
\label{1793616133929}
	\hbar^2 \del_x^2 \Psi + x^{-2}Q (x) \Psi = 0
\fullstop
\end{eqntag}
With $\psi = (\Psi, x \hbar \del_x \Psi)^\top$, this is equivalent to
\begin{eqn}
	x \hbar \del_x \psi +
	\mtx{ 0 & - 1 \\ Q (x) & - \hbar } \psi = 0
\fullstop
\end{eqn}
The classical residue has eigenvalues $\pm \sqrt{ - Q (0)}$, which are distinct.
The residue eigenvalues are the two roots of $\rho^2 + \hbar \rho + Q (0) = 0$, and the leading-order eigenvalues are $\pm \sqrt{ - Q (x)}$ on a sufficiently small disc.
Choose an ordering and an admissible direction.
The first coordinate of the $\eta_1$-eigenline is nonzero.
Thus, after decreasing the domains, the regular invariant line of \autoref{1793461788536} is uniquely the graph of a regular function $p$:
\begin{eqn}
	\L = \CC (1, p)^\top
\fullstop{,}
\qquad
	p (0, \hbar) = - \rho_1 (\hbar)
\fullstop{,}
\qquad
	p_0 (x) = - \eta_1 (x)
\fullstop
\end{eqn}
The condition that this graph be $\nabla$-invariant is precisely the Riccati equation
\begin{eqntag}
\label{1793616148028}
	x \hbar \del_x p = \hbar p - p^2 - Q
\fullstop
\end{eqntag}
The corresponding scalar solution, on a punctured sector with a chosen logarithm, is
\begin{eqntag}
\label{1793620873859}
	\Psi_1 (x, \hbar) = x^{ - \rho_1 (\hbar)/\hbar}
	\exp \left ( \frac{1}{\hbar}
	\int_0^x \frac{p (t, \hbar) + \rho_1 (\hbar)}{t}\, dt \right)
\fullstop
\end{eqntag}
The integrand extends holomorphically across zero because its numerator vanishes there.
The identity $x \hbar \del_x \Psi_1 = p \Psi_1$, together with \eqref{1793616148028}, verifies \eqref{1793616133929} directly.
For each fixed parameter, this solution has the Frobenius normalisation
\begin{eqn}
	\lim_{x \to 0} x^{\rho_1 (\hbar)/\hbar} \Psi_1 (x, \hbar) = 1
\fullstop
\end{eqn}
On the region \eqref{1793615393525}, it spans the subdominant solution line of the associated system.

The first two Riccati coefficients and the residue expansion are
\begin{eqn}
	p_0 = - \eta_1
\fullstop{,}
\qquad
	p_1 = \frac12 - \frac{x \del_x \eta_1}{2 \eta_1}
\fullstop{,}
\qquad
	\rho_1 = m_1 - \frac{\hbar}{2} + O (\hbar^2)
\fullstop
\end{eqn}
Substitution in \eqref{1793620873859} gives the matching formula
\begin{eqntag}
\label{1793620865580}
	\Psi_1 (x, \hbar) =
	x^{ - \rho_1 (\hbar)/\hbar}
	\left ( \frac{m_1}{\eta_1 (x)} \right)^{1/2}
	\exp \left ( - \frac{1}{\hbar}
	\int_0^x \frac{\eta_1 (t) - m_1}{t}\, dt \right)
	\big (1 + O (\hbar) \big)
\fullstop
\end{eqntag}
The square root is chosen to equal $1$ at $x = 0$.
After extracting the displayed singular factors, the relative error is uniform on a smaller disc in $x$ and sector in $\hbar$.
Indeed, the remainder of $p + \rho_1$ after its first two terms is $O (\hbar^2)$ and vanishes at $x = 0$.
Cauchy's estimates make it $x O (\hbar^2)$ on a smaller disc, so its contribution to the exponent in \eqref{1793620873859} is $O (\hbar)$ uniformly.
The remaining relative factor is itself regular and equals $1$ at $x = 0$.
Thus \eqref{1793620865580} combines the WKB phase and amplitude with the exact Frobenius exponent; see \cite{takei2017wkb,MY210623112236} for the WKB setting.
Retaining the exact exponent is essential for uniformity at the pole: its higher-order terms multiply $\log x$ when the power of $x$ is expanded in $\hbar$.
\end{example}

%---------------------------------------
\begin{example}[a boundary condition at the pole]
\label{1793620883357}
For a concrete instance, take $Q (x) = - 1 - x$, $\eta_1 (x) = - \sqrt{1 + x}$, $m_1 = - 1$, and $\theta = 0$, with square roots equal to $1$ at the origin.
Then
\begin{eqn}
	\rho_1 (\hbar) = - \frac{\hbar}{2} - \sqrt{1 + \hbar^2/4}
\fullstop{,}
\qquad
	F (x) = 2 \big (\sqrt{1 + x} - 1\big)
	- 2 \log \frac{\sqrt{1 + x} + 1}{2}
\fullstop
\end{eqn}
Since $F (0) = 0$ and $\del_x F = (\sqrt{1 + x} - 1)/x$, formula \eqref{1793620865580} becomes
\begin{eqn}
	\Psi_1 (x, \hbar) =
	x^{1/2 + \sqrt{1 + \hbar^2/4}/\hbar}
	(1 + x)^{-1/4} e^{F (x)/\hbar}
	\big (1 + O (\hbar) \big)
\fullstop
\end{eqn}
For sufficiently small positive $\hbar$, this is the unique solution up to scale that tends to zero as $x \to 0$ along the positive real axis.
Indeed, its Frobenius exponent $\alpha_+ = 1/2 + \sqrt{1 + \hbar^2/4}/\hbar$ is greater than $1$, and reduction of order gives an independent solution with leading term $x^{1 - \alpha_+}/(1 - 2 \alpha_+)$, which diverges.
Thus the invariant line selects a scalar boundary condition at the pole and gives its uniform asymptotics as $\hbar \to 0$.
\end{example}

\clearpage
\bibliographystyle{src/declare-bibliography}
\bibliography{src/bibliography}

\begin{thebibliography}{CDMFS07}

\bibitem[Bal94]{Balser1994}
W.~Balser, {\itshape From Divergent Power Series to Analytic Functions},
  vol.~1582 of {\itshape Lecture Notes in Mathematics}.
\newblock
  \DeclareBibliographyPublicationLink{10.1007/BFb0073564}{}{Springer-Verlag,
  Berlin}, 1994.

\bibitem[BK02]{MR1914445}
W.~Balser and V.~Kostov, {\itshape Singular perturbation of linear systems with
  a regular singularity},
  \DeclareBibliographyPublicationLink{10.1023/A:1016326320001}{}{J. Dynam.
  Control Systems {\bfseries 8} no.~3, (2002) 313--322}.

\bibitem[BMF02]{MR1919789}
W.~Balser and J.~Mozo-Fern\'{a}ndez, {\itshape Multisummability of formal
  solutions of singular perturbation problems},
  \DeclareBibliographyPublicationLink{10.1006/jdeq.2001.4143}{}{J. Differential
  Equations {\bfseries 183} no.~2, (2002) 526--545}.

\bibitem[Boa11]{MR2785493}
P.~P. Boalch, {\itshape Riemann-{H}ilbert for tame complex parahoric
  connections},
  \DeclareBibliographyPublicationLink{10.1007/s00031-011-9121-1}{}{Transform.
  Groups {\bfseries 16} no.~1, (2011) 27--50},
  \DeclareBibliographyArxiv{1003.3177}{math.DG}.

\bibitem[CDMFS07]{MR2292516}
M.~Canalis-Durand, J.~Mozo-Fern\'{a}ndez, and R.~Sch\"{a}fke, {\itshape
  Monomial summability and doubly singular differential equations},
  \DeclareBibliographyPublicationLink{10.1016/j.jde.2006.11.005}{}{J.
  Differential Equations {\bfseries 233} no.~2, (2007) 485--511}.

\bibitem[CDRSS00]{CRSS2000}
M.~Canalis-Durand, J.-P. Ramis, R.~Sch{\"a}fke, and Y.~Sibuya, {\itshape
  {Gevrey} solutions of singularly perturbed differential equations},
  \DeclareBibliographyPublicationLink{10.1515/crll.2000.008}{}{J. Reine Angew.
  Math. {\bfseries 518} (2000) 95--129}.

\bibitem[Lev61]{MR0145108}
A.~H.~M. Levelt, \DeclareBibliographyPublicationLink{}{}{{\itshape
  Hypergeometric functions}}.
\newblock PhD thesis, University of Amsterdam, 1961.

\bibitem[MC18]{MarchettiConti2018}
D.~H.~U. Marchetti and W.~R.~P. Conti, {\itshape Singular perturbation of
  nonlinear systems with regular singularity},
  \DeclareBibliographyPublicationLink{10.1155/2018/5163492}{}{Discrete Dyn.
  Nat. Soc. {\bfseries 2018} (2018) 5163492},
  \DeclareBibliographyArxiv{1207.4212}{math.CV}. 10 pages.

\bibitem[Nik23a]{MY2008.06492}
N.~Nikolaev, {\itshape Exact Solutions for the Singularly Perturbed {Riccati}
  Equation and Exact {WKB} Analysis},
  \DeclareBibliographyPublicationLink{10.1017/nmj.2022.38}{}{Nagoya Math. J.
  {\bfseries 250} (2023) 434--469},
  \DeclareBibliographyArxiv{2008.06492}{math.CA}.

\bibitem[Nik23b]{MY210623112236}
N.~Nikolaev, {\itshape Existence and Uniqueness of Exact {WKB} Solutions for
  Second-Order Singularly Perturbed Linear {ODEs}},
  \DeclareBibliographyPublicationLink{10.1007/s00220-022-04603-7}{}{Commun.
  Math. Phys. {\bfseries 400} no.~1, (2023) 463--517},
  \DeclareBibliographyArxiv{2106.10248}{math.AP}.

\bibitem[Nik24]{MY211216122156}
N.~Nikolaev, {\itshape Gevrey Asymptotic Implicit Function Theorem},
  \DeclareBibliographyPublicationLink{10.4171/LEM/1061}{}{Enseign. Math.
  {\bfseries 70} no.~1--2, (2024) 251--282},
  \DeclareBibliographyArxiv{2112.08792}{math.CV}.

\bibitem[RS66]{MR0214869}
D.~L. Russell and Y.~Sibuya, {\itshape The problem of singular perturbations of
  linear ordinary differential equations at regular singular points. {I}},
  \DeclareBibliographyPublicationLink{}{}{Funkcial. Ekvac. {\bfseries 9} (1966)
  207--218}.

\bibitem[RS68]{MR0245929}
D.~L. Russell and Y.~Sibuya, {\itshape The problem of singular perturbations of
  linear ordinary differential equations at regular singular points. {II}},
  \DeclareBibliographyPublicationLink{}{}{Funkcial. Ekvac. {\bfseries 11}
  (1968) 175--184}. Published in 1969.

\bibitem[Sib58]{MR0096016}
Y.~Sibuya, {\itshape Sur r{\'e}duction analytique d'un syst{\`e}me
  d'{\'e}quations diff{\'e}rentielles ordinaires lin{\'e}aires contenant un
  param{\`e}tre}, \DeclareBibliographyPublicationLink{}{}{J. Fac. Sci. Univ.
  Tokyo. Sect. I {\bfseries 7} (1958) 527--540}.

\bibitem[Sib00]{Sibuya2000}
Y.~Sibuya, {\itshape The {Gevrey} asymptotics in the case of singular
  perturbations},
  \DeclareBibliographyPublicationLink{10.1006/jdeq.2000.3787}{}{J. Differential
  Equations {\bfseries 165} no.~2, (2000) 255--314}.

\bibitem[Tak17]{takei2017wkb}
Y.~Takei, {\itshape {WKB} analysis and {Stokes} geometry of differential
  equations}, in {\itshape Analytic, Algebraic and Geometric Aspects of
  Differential Equations}, pp.~263--304.
\newblock \DeclareBibliographyPublicationLink{}{}{Springer}, 2017.

\bibitem[Was76]{MR0460820}
W.~Wasow, {\itshape Asymptotic expansions for ordinary differential equations}.
\newblock \DeclareBibliographyPublicationLink{}{}{Robert E. Krieger Publishing
  Co., Huntington, N.Y.}, 1976.
\newblock Reprint of the 1965 edition.

\bibitem[{\.Z}o{\l}06]{zoladek2006monodromy}
H.~{\.Z}o{\l}{\k{a}}dek, {\itshape The Monodromy Group}, vol.~67 of {\itshape
  Monografie Matematyczne}.
\newblock \DeclareBibliographyPublicationLink{}{}{Birkh{\"a}user Basel}, 2006.

\end{thebibliography}
\end{document}